\documentclass[12pt]{amsart}

\newif\ifPDF
\ifx\pdfoutput\undefined\PDFfalse
\else \ifnum \pdfoutput > 0 \PDFtrue
        \else \PDFfalse
        \fi
\fi

\usepackage[centertags]{amsmath}
\usepackage{amsfonts}
\usepackage{mathrsfs}
\usepackage{textcomp}
\usepackage{amssymb}
\usepackage{amsthm}
\usepackage{newlfont}
\usepackage[all]{xy}

\ifPDF
  \usepackage[pdftex]{color, graphicx}
  \usepackage[pdftex, bookmarks, colorlinks]{hyperref}
  \hypersetup{colorlinks=false}

\else
  \usepackage{color}
  \usepackage[dvips]{graphicx}
  \usepackage[dvips]{hyperref}
\fi

\usepackage[margin=1in]{geometry}

\usepackage[pagewise, mathlines, displaymath]{lineno}

\newtheorem{thm}{Theorem}[section]
\newtheorem{cor}[thm]{Corollary}
\newtheorem{lem}[thm]{Lemma}
\newtheorem{prop}[thm]{Proposition}
\theoremstyle{definition}
\newtheorem{defn}[thm]{Definition}
\theoremstyle{remark}

\numberwithin{equation}{section}

\newcommand{\norm}[1]{\left\Vert#1\right\Vert}
\newcommand{\abs}[1]{\left\vert#1\right\vert}

\newcommand{\Int}{\mathbb Z}

\begin{document}


\title[Dynamical Asymptotic Dimension]{On the Dynamical Asymptotic Dimension of a free $\Int^d$-action on the Cantor set}

\author{Zhuang Niu}
\address{Department of Mathematics, University of Wyoming, Laramie, Wyoming, USA, 82071}
\email{zniu@uwyo.edu}

\author{Xiaokun Zhou}
\address{Department of Mathematics, University of Wyoming, Laramie, Wyoming, USA, 82071}
\email{xzhou3@uwyo.edu}

\thanks{The research is supported by an NSF grant (DMS-1800882).}
\keywords{Dynamical asymptotic dimension, free $\mathbb Z^d$-actions, Cantor system}
\date{\today}


\begin{abstract}
Consider an arbitrary extension of a free $\Int^d$-action on the Cantor set. It is shown that it has dynamical asymptotic dimension at most $3^d - 1$.
\end{abstract}

\maketitle

\section{Introduction}

Dynamical Asymptotical Dimension is introduced by Guentner, Willett, and Yu in \cite{GWY} to describe the complexity of a topological dynamical system: 
\begin{defn}
Consider a group action $X\curvearrowleft \Gamma$, where $X$ is a compact Hausdorff space and $\Gamma$ is a discrete group. Its dynamical asymptotic dimension (DAD) is the smallest non-negative integer $d$ such that for any finite subset $\mathcal F \subseteq \Gamma$, there is an open cover $U_0 \cup U_1 \cup \cdots \cup U_d$ of $X$ such that for each $U_i$, $0\leq i\leq d$, each $x\in U_i$,  the cardinality of the set $$\mathcal O_x:=\{y\in U_i: \exists \gamma_1, ..., \gamma_K\in \mathcal F,\ y=x\gamma_1\cdots\gamma_K,\ x\gamma_1\cdots\gamma_k\in U_i,\ 1\leq k\leq K,\ K\in\mathbb N\}$$ is finite and uniformly bounded (with respect to $x$).
\end{defn}

It is shown in \cite{GWY} that the dynamical asymptotical dimension of any free $\Int$-action is at most $1$, regardless of the space $X$. It is also shown in \cite{GWY} that for any discrete group $\Gamma$ with asymptotic dimension at most $d$, there is a $\Gamma$-action on the Cantor set which has  dynamical asymptotical dimension at most $d$. In this note, we estimate the dynamical asymptotical dimension of an arbitrary $\Int^d$-action on the Cantor set. In fact, we have the following theorem:
%
\theoremstyle{theorem}
\newtheorem*{thmA}{Theorem}
\begin{thmA}[Theorem \ref{main-thm} and Corollary \ref{main-cor}]
Any extension of a free $\Int^d$-action on the Cantor set has dynamical asymptotic dimension at most $3^d-1$.
\end{thmA}

\section{Main result and its proof}

\subsection{Quasi-tilings of $\Int^d$}
Let us start with certain quasi-tilings (see \cite{OW-tiling}) of $\Int^d$ by cubes:
\begin{defn}
Consider $\Int^d$. For any natural number $l$, denote by $\square_l$ the cube $$\square_l=\{-l, -l+1, ..., l-1, l\}^d\subseteq \Int^d.$$

Let $r, D, E$ be natural numbers. An $(r, D, E)$-tiling of $\Int^d$, denoted by $\mathcal T$, is a collection of $c_i\in \Int^d$ such that with  
$$\mathrm{Dom}(\mathcal T) = \bigcup_{i}(c_i + \square_D),$$ then,
\begin{enumerate}
\item $(c_i + \square_D) \cap (c_j+\square_D) = \varnothing$, $i\neq j$,
\item The (Euclidean) distance between $c_i + \square_D$ and $c_j+\square_D$ is at least $r$ if $i\neq j$, and
\item $\square_E \cap \mathrm{Dom}(\mathcal T) \neq \varnothing$.
\end{enumerate}
\end{defn}

In other words, an $(r, D, E)$-tiling of $\Int^d$ is a quasi-tiling by cubes of size $2D+1$, such that tiles are $r$-separated, but they almost cover $0$ up to $E$. 

It turns out that if $D\leq E\leq 2D$, then there are $e_0=0, e_1, e_2, ..., e_{3^d-1}\in \Int^d$ such that for any $(r, D, E)$-tiling $\mathcal T$, one of $\mathcal T, \mathcal T+e_1, ..., \mathcal T+ e_{3^d-1}$ actually covers $0$:

\begin{lem}\label{shift-tiling}
For any natural number $E$, then there are $e_1, e_2, ..., e_s\in \Int^d$, where $s=3^d-1$, such that if $\mathcal T$ is an $(r, D, E)$-tiling of $\Int^d$ for some natural numbers $r$ and $D$ with $D\leq E \leq 2D$, then $$ 0 \in \mathrm{Dom}(\mathcal T) \cup \mathrm{Dom}(\mathcal T + e_1)\cup\cdots \cup \mathrm{Dom}(\mathcal T + e_s),$$ where $s = 3^d-1$.
\end{lem}

\begin{proof}
Set $$\{e_0, e_1, ..., e_{3^d-1}\}=\{(n_1, n_2, ..., n_d) \in \Int^d: n_i \in \{0, \pm E\}\},$$ with $e_0=(0, ..., 0)$. In order to prove the lemma, it is enough to show that if $0 \notin \mathrm{Dom}(\mathcal T)$, then, at least one of $$e_i,\quad i=1, ..., 3^d-1,$$ is in $\mathrm{Dom}(\mathcal T)$.

Assume none of $e_i$ was inside $\mathrm{Dom}(\mathcal T)$. Then one asserts that $$\square_E \cap \mathrm{Dom}(\mathcal T) = \varnothing.$$ This contradicts Condition (3) and hence proves the lemma.

For the assertion, assume there is $c\in\Int^d$ with $$c+\square_D \subseteq \mathrm{Dom}(\mathcal T)\quad \textrm{and}\quad \square_E \cap (c+\square_D) \neq \varnothing.$$ Then there exist $$-E \leq n_i\leq  E,\quad 1\leq i\leq d,$$ such that $$(n_1, ..., n_d) \in c+\square_D.$$ Note that $\square_E \cap (c+\square_D) \neq \varnothing$ implies $$-D-E \leq c_i \leq D+E,\quad 1\leq i\leq d,\ c=(c_1, c_2, ..., c_d);$$ 
and also note  
$$c+\square_D = \{(c_1+s_1, c_2+s_2, ..., c_d+s_d): -D \leq s_i \leq D \}.$$ For each $c_i$, if $\abs{c_i}\geq E$, then choose $s_i \in [-D, D]$ such that $\abs{c_i + s_i} = E$; if $\abs{c_i}\leq D$, then choose $s_i =- c_i $ so that $c_i + s_i = 0$; if $D\leq \abs{c_i} \leq E$, then choose $s_i\in [-D, D]$ such that $\abs{c_i + s_i} = E$ (note that one assumes $E\leq 2D$). With this choice of $s_i$, one has that $c+ \square_D$ contains at least one of $e_i$, and so such $e_i$ is inside $\mathrm{Dom}(\mathcal T)$. This contradicts the assumption, and proves the assertion.
\end{proof}

\subsection{Group actions and equivariant quasi-tilings}

Recall
\begin{defn}
Let $X$ be a topological space and let $\Gamma$ be a discrete group. By a (right) $\Gamma$-action on $X$, denoted by $X\curvearrowleft \Gamma$, we mean a continuous map $$X \times \Gamma \ni (x, \gamma) \to x\gamma \in X$$ such that $$ xe=x\quad\mathrm{and}\quad (x\gamma_1)\gamma_2 = x(\gamma_1\gamma_2),\quad x\in X,\ \gamma_1\gamma_2\in \Gamma.$$ 

We say a $\Gamma$-action on $X$ is free if $x\gamma = x$ for some $x\in X$ and $\gamma\in\Gamma$ implies $\gamma = e$.

Consider actions $X\curvearrowleft \Gamma$ and $Y\curvearrowleft \Gamma$. We say that $X\curvearrowleft \Gamma$ is an extension of $Y\curvearrowleft \Gamma$ (or $Y\curvearrowleft \Gamma$ is a factor of $X\curvearrowleft \Gamma$) if there is a quotient map $\pi: X \to Y$ such that $$\pi(x\gamma) = \pi(x)\gamma,\quad x\in X,\ \gamma\in \Gamma.$$
\end{defn}

\begin{defn}
Consider an $\Int^d$-action on topological space $X$. A set-valued map $$X \ni x \mapsto \mathcal T(x) \in 2^{\Int^d}$$ is said to be equivariant if $$\mathcal T(xn) = \mathcal T(x) - n,$$ where $\mathcal T(x) - n$ is the translation of $\mathcal T(x)$ by $-n$.

The map $x \mapsto \mathcal T(x)$ is said to be continuous if for any $R>0$ and any $x\in X$, there is an open set $U \ni x$ such that $$\mathcal T(y) \cap B_R = \mathcal T(x) \cap B_R,\quad y\in U,$$
where $B_R$ is the ball in $\Int^d$ with center $0$ and radius $R$.
\end{defn}

\begin{lem}\label{domain-finite}
Consider an $\Int^d$-action on a topological space $X$. Let $N\in \mathbb N$, and let $x \mapsto \mathcal T(x)$ be a continuous equivariant map with value $(r, D, E)$-tilings of $\Int^d$ with $r > N\sqrt{d}$. Put $$\Omega = \{x \in X: 0 \in \mathrm{Dom}(\mathcal T(x))\}.$$  Then, $\Omega$ is open. Moreover, for any $x\in X$, one has
\begin{eqnarray}\label{bd-Omega}
& & \left | \{n \in \Int^d: n=n_1+\cdots+n_K,\ x(n_1+\cdots+n_k) \in\Omega, \norm{n_k}_\infty \leq N,\right.\\
& & \hskip 1mm \left. 1\leq k \leq K,\ K\in\mathbb N \} \right | \nonumber \\
& \leq & (2D+1)^d. \nonumber
\end{eqnarray}
\end{lem}

\begin{proof}
The openness of $\Omega$ follows directly from the continuity of the map $x\mapsto\mathcal T(x)$. Let us show the estimate \eqref{bd-Omega}.

Pick $x_0\in \Omega$, and write $c+\square_D$ to be the tile of $\mathcal T(x_0)$ containing $0$. Since the function $x \mapsto \mathcal T(x)$ is equivariant, one has that $\mathcal{T}(xn) = \mathcal{T}(x) - n$; hence, by Condition (2), for any $n\in \Int^d$ with $\norm{n}_\infty \leq N$, one has that either $0$ is in the tile $c + \square_D - n$ (therefore $x_0n \in \Omega$ and $c-n \in \square_D$) or $0\notin \mathrm{Dom}(\mathcal T(x_0n))$ (therefore $x_0n \notin\Omega$). 

Thus, if there are $n_1, n_2, ..., n_K \in\Int^d$ with $\norm{n_k}_\infty \leq N$ and $$n_1x_0\in\Omega,\ x_0(n_1+n_2)\in\Omega, ..., x_0(n_1+\cdots+n_K) \in \Omega,$$ one has
$$c-n_1\in\square_D,\ c-n_1-n_2\in \square_D, ..., c-n_1-\cdots - n_K \in \square_D, $$ and hence $$n=n_1+\cdots+n_K \in c+\square_D.$$
Since $\abs{c+\square_D} = \abs{\square_D} = (2D+1)^d$, this proves the lemma.
\end{proof}

\subsection{Cantor systems and an estimate of dynamical asymptotic dimension}
Let us focus on extensions of a free $\Int^d$-action on the Cantor set, which is the unique compact separable Hausdorff space that is totally disconnected and perfect.

First, for any free $\Int^d$-action on the Cantor set, equivariant continuous $(r, D, E)$-tiling-valued functions always exist:
\begin{prop}\label{exist-tiling}
Consider a free $\Int^d$-action on $X$ where $X$ is the Cantor set, and let $N\in\mathbb N$ be arbitrary. Then, there are natural numbers $r, D, E$ with $r>N\sqrt{d}$ and $D\leq E \leq 2D$, and a continuous equivariant map $x \mapsto \mathcal T(x)$ on $X$ such that each $\mathcal T(x)$ a $(r, D, E)$-tiling of $\Int^d$.
\end{prop}

\begin{proof}
The construction is similar to that of Lemma 3.4 of \cite{DH-tiling}.

Pick a natural number $r> N\sqrt{d}$, and then pick a natural number $L > 2r$. Since the action is free and $X$ is the Cantor set, by a compactness argument, one obtains mutually disjoint clopen sets $U_1, U_2, ..., U_s$, such that $$X=U_1 \cup U_2\cup\cdots\cup U_s,$$ and for each $U_i$, $1\leq i \leq s$, the open sets $$U_in,\quad n\in \square_{2L}, $$ are mutually disjoint.

Start with $U_1$. For each $x\in X$, put 
$$\left\{
\begin{array}{lll}
\mathcal C_1(x)  & = & \{n\in\Int^d: xn \in U_1\}, \\
\cdots & \cdots & \cdots \\
\mathcal C_i(x) & = & \mathcal C_{i-1}(x) \cup \{n\in\Int^d: xn \in U_i,\ (n+ \square_L) \cap (\mathcal C_{i-1}(x) + \square_L) = \varnothing \},\\
\cdots &\cdots & \cdots \\
\mathcal C_s(x) & = & \mathcal C_{s-1}(x) \cup \{n\in\Int^d: xn \in U_s,\ (n+ \square_L) \cap (\mathcal C_{s-1}(x) + \square_L) = \varnothing \}.
\end{array}
\right.$$

Since $U_1$ is clopen, the map $x\mapsto \mathcal C_1(x)$ is continuous in the sense that for any $x$ and any $R>0$, there is a neighbourhood $W$ of $x$ such that $$\mathcal C_1(y) \cap B_R = \mathcal C_1(x) \cap B_R,\quad y\in W.$$

Consider the map $x\mapsto \mathcal C_2(x)$. Fix $x\in X$, $R>0$. Since $U_2$ is clopen, there is a neighbourhood $W$ of $x$ such that 
$$\{n\in \Int^d: xn\in U_2\} \cap B_R = \{n\in \Int^d: yn\in U_2\} \cap B_R,\quad y\in W.$$ Note that $x\mapsto\mathcal C_1(x)$ is continuous, then the neighbourhood $W$ can be chosen so that $$(\mathcal C_{1}(x) + \square_L) \cap B_R = (\mathcal C_{1}(x) + \square_L)\cap B_R, \quad y\in W,$$ and therefore for any $y\in W$,
\begin{eqnarray*}
& & \{xn \in U_2,\ (n+ \square_L) \cap (\mathcal C_{1}(x) + \square_L) = \varnothing\} \cap B_R  \\
& = & \{yn \in U_2,\ (n+ \square_L) \cap (\mathcal C_{1}(y) + \square_L) = \varnothing\}\cap B_R.
\end{eqnarray*} 
Together with the continuity of $x \mapsto \mathcal C_1(x)$, this shows that $x\mapsto \mathcal C_2(x)$ is continuous.

Repeat this argument, one shows that the map $x\mapsto \mathcal C_s(x)$ is continuous.

Let us show that the map $x\mapsto \mathcal C_s(x)$ is equivariant. Start with $x \mapsto \mathcal C_1(x)$. Let $n\in\Int^d$ and consider $xn$. Since $xm \in U_1$ if and only if $x(n+m-n)\in U_1$, one has $$\mathcal C_1(xn) = \mathcal C_1(x) - n.$$ A similar argument shows that $\mathcal C_2(x), ..., \mathcal C_s(x)$ are equivariant.

One asserts that $$(c_1+\square_L)\cap(c_2+\square_L) = \varnothing,\quad c_1\neq c_2,\ c_1, c_2 \in \mathcal C_s(x).$$ 

Indeed, since $U_1n$, $n\in \square_{2L}$,  are mutually disjoint, one has that $$(c+\square_{2L}) \cap \mathcal C_1(x) = c,\quad c\in\mathcal C_1(x),$$ and thus $$(c_1+\square_L)\cap(c_2+\square_L) = \varnothing,\quad c_1\neq c_2,\ c_1, c_2 \in \mathcal C_1(x).$$ 

Now, pick $$c_1, c_2\in \mathcal C_2(x) =  \mathcal C_{1}(x) \cup \{n\in\Int^d: xn \in U_2,\ (n+ \square_L) \cap (\mathcal C_{1}(x) + \square_L) = \varnothing \}.$$ If $c_1, c_2 \in \mathcal C_1(x)$, then as shown above, $$(c_1+\square_L)\cap(c_2+\square_L) = \varnothing.$$
Assume that $$c_1, c_2\in \{n\in\Int^d: xn \in U_2,\ (n+ \square_L) \cap (\mathcal C_{1}(x) + \square_L) = \varnothing \}\subseteq \{n\in\Int^d: xn \in U_2 \}.$$ Then, since $U_2n$, $n\in \square_{2L}$,  are mutually disjoint, the same argument as that of $\mathcal C_1(x)$ shows that $$(c_1+\square_L)\cap(c_2+\square_L) = \varnothing.$$ 
Assume that $c_1\in\mathcal C_1$ and $c_2\in \{n\in\Int^d: xn \in U_2,\ (n+ \square_L) \cap (\mathcal C_{1}(x) + \square_L) = \varnothing \}$. Then the equation $$(c_1+\square_L)\cap(c_2+\square_L) = \varnothing$$ just follows from the definition. 

Repeat this argument for $\mathcal C_3(x), ...,\mathcal C_s(x)$, and this proves the assertion.

Note that for the given $x$, there exists a $U_i$ containing $x$. Therefore, either $$\square_L \cap (\mathcal C_{i-1}(x) + \square_L) \neq \varnothing\quad \textrm{or}\quad 0\in \mathcal C_i(x).$$ In particular, one always has that $\square_L \cap (\mathcal C_i(x) + \square_L) \neq \varnothing$, and hence
$$\square_L \cap (\mathcal C_s(x) + \square_L) \neq \varnothing.$$
 
 
To summarize, setting $\mathcal C(x)=\mathcal C_s(x)$, one obtains a continuous equivariant map $x \mapsto \mathcal C(x)$ satisfying
 \begin{enumerate}
\item[(1)] $(c_i + \square_L) \cap (c_j+\square_L) = \varnothing$, $c_i\neq c_j$, $c_i, c_j\in\mathcal C_s(x)$ and
\item[(2)] $\square_L \cap (\mathcal C_s(x) + \square_L) \neq \varnothing$;
\end{enumerate}
hence it satisfies
 \begin{enumerate}
\item[(3)] $(c_i + \square_{L-r}) \cap (c_j+\square_{L-r}) = \varnothing$, $c_i\neq c_j$, $c_i, c_j\in\mathcal C_s(x)$,
\item[(4)] $\square_{L+r} \cap (\mathcal C_s(x) + \square_{L-r}) \neq \varnothing$;
\end{enumerate}
and, moreover
\begin{enumerate}
\item[(5)] the (Euclidean) distance between $c_i + \square_{L-r}$ and $c_j+\square_{L-r}$ is at least $r$ if $c_i\neq c_j$.
\end{enumerate}
Thus, each $\mathcal C(x)$ is an $(r, L-r, L+r)$ tiling. Since $L > 2r$, one has $L+r < 2(L-r)$, and this proves the statement of the proposition.
\end{proof}

\begin{cor}\label{tiling-cover}
Consider a free $\Int^d$-action on $X$ where $X$ is the Cantor set, and let $N\in\mathbb N$ be arbitrary.
Then, there exist continuous equivariant maps $$x \mapsto \mathcal T_i(x),\quad i=0, 1, ..., 3^d-1,$$ with each $\mathcal T_i(x)$ a $(r, D, E)$-tilings of $\Int^d$ for some $r, D, E\in\mathbb N$ with $r > N\sqrt{d}$, such that, if put $$\Omega_i = \{x \in X: 0 \in \mathrm{Dom}(\mathcal T_i(x))\},\quad i=0, 1, ..., 3^d-1,$$ then
$$\Omega_0 \cup \Omega_1 \cup\cdots\cup\Omega_{3^d-1} = X.$$
\end{cor}

\begin{proof}
It follows from Proposition \ref{exist-tiling} that there are natural numbers $r, D, E$ with $$r>N\sqrt{d}\quad\mathrm{and}\quad D\leq E \leq 2D,$$ and a continuous equivariant map $x \mapsto \mathcal T_0(x)$ on $X$ such that each $\mathcal T_0(x)$ a $(r, D, E)$-tiling of $\Int^d$.

Consider the translations of the function $\mathcal T_0$:
$$\mathcal T_1=\mathcal T_0 + e_1,\ \mathcal T_2=\mathcal T_0 + e_2,\ ..., \mathcal T_{3^d-1}=\mathcal T_0 + e_{3^d-1},$$
where $e_1, ..., e_{3^d-1}$ are the vectors (with repect to $E$) obtained from Lemma \ref{shift-tiling}. Since $D \leq E \leq 2D$, it follows from Lemma \ref{shift-tiling} that for any $x\in X$, one has
$$0 \in \mathrm{Dom}(\mathcal T_0(x)) \cup  \mathrm{Dom}(\mathcal T_1(x)) \cup \cdots\cup  \mathrm{Dom}(\mathcal T_{3^d-1}(x)),$$ and thus
$$\Omega_0 \cup \Omega_1 \cup\cdots\cup\Omega_{3^d-1} = X,$$
as desired.
\end{proof}

\begin{thm}\label{main-thm}
The dynamical asymptotic dimension of any free $\Int^d$-action on the Cantor set is at most $3^d-1$.
\end{thm}
\begin{proof}
Let $N\in\mathbb N$ be arbitrary.
It follows from Corollary \ref{tiling-cover} that there exist continuous equivariant maps $$x \mapsto \mathcal T_i(x),\quad i=0, 1, ..., 3^d-1,$$ with each $\mathcal T_i(x)$ a $(r, D, E)$-tilings of $\Int^d$ for some $r, D, E$ with $r > N\sqrt{d}$ with
$$\Omega_0 \cup \Omega_1 \cup\cdots\cup\Omega_{3^d-1} = X,$$
where $$\Omega_i = \{x \in X: 0 \in \mathrm{Dom}(\mathcal T_i(x))\},\quad i=0, 1, ..., 3^d-1,$$ which is open.

Since $r> N\sqrt{d}$, by Lemma \ref{domain-finite}, for any $i=0, 1, ..., 3^d$, one has
\begin{eqnarray*}
& & \left |n \in \Int^d: n=n_1+\cdots+n_K,\ x(n_1+\cdots+n_k) \in\Omega_i,\ \norm{n_k}_\infty \leq N,\right. \\
& & \hskip 1mm \left. 1\leq k \leq K,\ K\in\mathbb N \} \right | \\
& \leq & (2D+1)^d<+\infty. \nonumber
\end{eqnarray*}
That is, the dynamical asymptotic dimension of $X\curvearrowleft \Int^d$ is at most $3^d-1$.
\end{proof}

\begin{lem}
Let $X\curvearrowleft \Gamma$ be an extension of a free action $Y\curvearrowleft \Gamma$. Then the dynamical asymptotic dimension of $X\curvearrowleft \Gamma$ is at most the dynamical asymptotic dimension of $Y\curvearrowleft \Gamma$.
\end{lem}
\begin{proof}

Let $d\in\Int$ such that the dynamical asymptotical dimension of $Y\curvearrowleft \Gamma$ is at most $d$. Let $\Gamma_0\subseteq \Gamma$ be finite. Then, together with the freeness of $Y\curvearrowleft \Gamma$, there exist an open cover $U_0 \cup U_1 \cup \cdots \cup U_d$ of $Y$ and $M>0$ such that for each $U_i$, $0\leq i\leq d$, $y_0\in U_i$, one has that 
\begin{equation}\label{bd-Y}
\abs{\{\gamma_1\cdots\gamma_K: \exists \gamma_1, ..., \gamma_K\in \Gamma_0,\ y_0\gamma_1\cdots\gamma_k\in U_i,\ 1\leq k\leq K,\ K\in \mathbb N\}} \leq M.
\end{equation}

Consider the open sets $$\pi^{-1}(U_0),\ \pi^{-1}(U_1),\ ...,\ \pi^{-1}(U_d),$$ where $\pi: X \to Y$ is the quotient map, and note that they form an open cover of $X$. For each $0\leq i\leq d$, pick an arbitrary $x_0\in \pi^{-1}(U_i)$ and assume there are $\gamma_1, ..., \gamma_K\in\Gamma_0$ for some $K\in \mathbb N$ such that 
$$x_0\in \pi^{-1}(U_i),\ x_0\gamma_1 \in \pi^{-1}(U_i),\ ... ,\  x_0\gamma_1\gamma_2\cdots\gamma_K \in \pi^{-1}(U_i).$$ Applying the quotient map $\pi$, one has
$$\pi(x_0)\in U_i,\ \pi(x_0)\gamma_1 \in U_i,\ ... ,\  \pi(x_0)\gamma_1\gamma_2\cdots\gamma_K \in U_i,$$ and, by \eqref{bd-Y}, this implies
\begin{equation*}
\abs{\{\gamma_1\cdots\gamma_K: \exists \gamma_1, ..., \gamma_K\in \Gamma_0,\ x_0\gamma_1\cdots\gamma_k\in \pi^{-1}(U_i),\ 1\leq k\leq K,\ K\in\mathbb N\}} \leq M.
\end{equation*}
Thus, the dynamical asymptotic dimension of $X\curvearrowleft \Gamma$ is at most $d$. 
\end{proof}

Then, the following is a straightforward corollary of Theorem \ref{main-thm}:
\begin{cor}\label{main-cor}
The dynamical asymptotic dimension of any extension of a free $\Int^d$-action on the Cantor set is at most $3^d-1$.
\end{cor}

\bibliographystyle{plainurl}
\bibliography{operator_algebras}

\end{document}